\documentclass{article}

\usepackage[utf8]{inputenc}
\usepackage{amssymb}
\usepackage{amsmath}
\usepackage{amsthm}
\usepackage[cjk]{kotex}
\usepackage{hyperref}
\usepackage{multicol}

\theoremstyle{theorem}
\newtheorem{theorem}{Theorem}

\theoremstyle{definition}

\begin{document}

\title{Homophonic Quotients of Linguistic Free Groups: German, Korean, and Turkish}
\markright{Homophones in German, Korean, and Turkish}
\author{Herbert Gangl \and Gizem Karaali \and Woohyung Lee}
\date{}
\maketitle

\begin{abstract}
In 1993, the homophonic quotient groups for French and English (the quotient of the free group generated by the French (respectively English) alphabet determined by relations representing standard pronunciation rules) were explicitly characterized \cite{QHGL}. In this paper we apply the same methodology to three different language systems:\ German, Korean, and Turkish. We argue that our results point to some interesting differences between these three languages (or at least their current script systems). 
\end{abstract}

\section{Introduction.}

In 1993, the homophonic quotient groups for French and English (the quotient of the free group generated by the French (respectively the English) alphabet determined by relations representing standard pronunciation rules) were explicitly characterized \cite{QHGL}.
Some references mention an analogous characterization for Japanese, but that result does not seem to be easily accessible. 

In this paper we apply the same methodology to three different language systems:\ German, Korean, and Turkish. The analysis for German was circulated in unpublished form for a while; the Korean and the Turkish analyses are new. As we suggest in the final section of this paper, our results may point to some interesting differences between these three languages (or at least their current script systems). 

The paper is organized in a straightforward manner, with each numbered section presenting the analysis for one language. In particular Section \ref{S:German} presents our results for 
German, while Section \ref{S:Korean} presents our results for 
Korean, and Section \ref{S:Turkish} presents our results for
Turkish. A final section brings together these analyses and offers some thoughts on what we might gain from this comparative study. 

\section{German}
\label{S:German}

In their phonetically calibrated paper \cite{QHGL}, Mestre, Schoof, Washington, and Zagier have shown that the homophonic quotient of the free group on the 26 letters of the alphabet is trivial for both the French and the English language. As already foreshadowed in that paper, we obtain the same answer for the German language. 

\smallskip
Let $G$ be the quotient of the free group on 26 letters $a,\,b,\,c,\,\dots,z$ by the relations
$A=B$ provided there are words $A$ and $B$ in the German language whose pronunciations agree. 

\medskip
We justify the term ``agree'' by invoking standard dictionaries like \cite{Duden1} or \cite{Duden2}, whose name ``Duden'' has become synonymous with the official norm, as well as its online version \url{http://www.duden.de/suchen/dudenonline}, accessed last on December 18, 2016. Alternatively, for most of the pairs of words below, we can use an automatic phonetic converter such as the one at \url{http://familientagebuch.de/rainer/2007/38.html#4}, accessed last on December 18, 2016.

\begin{theorem} The group $\,G\,$ is trivial.
\end{theorem}

\smallskip\noindent
{\bf Proof:}
We successively eliminate letters using specific properties of spoken German.
For homophonicity we need to distinguish in particular between long and short vowels as well as between voiced and unvoiced consonants.

\medskip\noindent
{\bf Vowels} (methods of idempotents (cf.~\cite{QHGL}) and of vanishing with h).
\begin{enumerate}
\item [(a)] For instance, `aa', `ah' and `a' may often be pronounced alike, in particular they often have the same length, like ``Waage'' [scales] and ``wage'' [(I) dare] or ``Wahl'' [choice] and ``Wal'' [whale].
\item [(e)] Similarly, `ee', `eh' and sometimes `e' can sound the same:  ``Meer'' [the sea] and ``mehr'' [more].
\item [(o)] Both `oo' and `oh' are often used within words, and can both be pronounced like a single `o' (``Boot'' [boat] and ``bot'' [(he) offered], and ``hohle'' [hollow (pl.)] vs.~``hole'' [(I) fetch]). 
\end{enumerate}
We note that for `i' and `u'  the corresponding identifications do not work; e.g., while both `ie' and `ih' indicate a long `i', the former can never occur at the beginning of a word where it is instead replaced by the second one (``ihnen'' [(to) them], ``ihr'' [her]), and `ii' in a word (like ``liieren'' [(to) liaise]) is pronounced with a glottal stop between the `i's.
Similarly, the `uu' in words like ``Kontinuum'' or ``Trauung'' indeed comes across as two `u's,
 and there aren't any words with a `uu' that would sound like a long `u', say.  Hence we need to treat these two vowels separately.

\bigskip\noindent
{\bf Consonants}.
\begin{enumerate}
\item[(g/b/n)] (Voiceless in the end.) At the end of a word, a voiced consonant is pronounced in the same way as the corresponding unvoiced one (like ``Bug'' [(nautical) bow] and ``buk'' [(he) baked], or ``Alb'' and ``Alp'' [both for nightmare]). Similarly, an `nn' at times sounds like a single `n' (``Mann'' [man] and ``man'' [one/you (pronoun)]).

\item[(v/w)] (WVF?) The consonant `v' is typically pronounced in one of two ways: like `f' or like `w', depending mostly on the etymological origin of the word (``viel'' [many] vs. ``fiel'' [(he) fell] and ``vage'' [vague] vs. ``wage'' [(I) dare]).

\item[(l/r/f/p/s)] (Idempotents.) By combining certain consonants we can further minimise the influence of a single contributing consonant, so while it is hard to find the same sounds for `ll' and `l' at the end of a word, one can add a `t' to it and succeeds (``hallt'' [(it) reverberates/echoes] vs. ``Halt'' [halt]). Similar comments apply to `rr' and `r' (``starrt'' [(he) stares] vs. ``Start'' [start]), for `ff' and `f' (``schafft'' [(he) manages] vs. ``Schaft'' [shaft]) as well as `pp' and `p' (``klappst'' [(you) flap/fold] vs. `klapst' [(he) claps lightly]; alternatively, ``schnippst'' and ``schnipst'' (both from schnipsen [(to) clip])).  Furthermore, ``fasst'' [(he) catches] and ``fast'' [almost] are homophonic.
\item[(t/d)] (Little dt for tt.) A related case is the combination `th' which also often assures that a preceding vowel is pronounced as a short one: e.g. ``Zithern'' [zithers] and ``zittern'' [(to) tremble] are pronounced the same way; another means to the same end is the use of `dt' in place of `tt', giving e.g. that ``Stadt'' [city] and ``statt'' [instead of] are homophonic.
\item[(m)]
A variant of the idempotent method, using also the voiced/unvoiced consonant at the end of a word, is ``hemmt'' [hinders] vs. ``Hemd'' [shirt].
\item[(c)] (Departing of the c.) Other constructs that make sure that a vowel is short are to follow it up with a `ck' rather than a `k', and the words ``packt'' [(he) packs] and ``Pakt'' [(a) pact] sound alike. 
Note, however, that in a very similar setting the words ``hackt'' [(he) hacks] and ``hakt'' [(he) hooks] are pronounced differently, as the latter `a' then denotes a {\em long} vowel.
\item[(z)] A further peculiarity is the pronunciation of `z', typically equivalent to the combination `t-s' (with obvious exceptions for loanwords like ``Jazz'' where the educated citizen will make an attempt to sound more anglophonic), so we can identify the genitive ``Kitts'' of ``Kitt''  [glue] with ``Kitz'' [fawn].
\item[(x)] In the same vein as `z', the letter `x' is pronounced `k-s' which is also the pronunciation of `chs', i.e. when `ch' precedes `s' it often becomes `k'); so we find ``lax'' [lax] to be homophonic to ``Lachs'' [salmon].
\end{enumerate}

The remaining letters `k', `u', `i', `y', `j' and `q' are somewhat harder to trivialise, but modulo the above this is doable, albeit by using loanwords from different languages (English, Italian, Hungarian).

\begin{enumerate}
\item[(k)] The English word ``Clip'' for office equipment is often used  and is homophonic to ``klipp'' (e.g.~from ``klipp und klar'' [in no uncertain terms]).
\item[(u)] The Italian word ``ciao'' has been assimilated as ``tschau'', both terms being used.
\item[(i)] The word ``roien'' [(to) row] is homophonic to ``reuen'' [(to) rue], the former being used mainly in ``Niederdeutsch'', i.e.~in the North of Germany.  Alternatively, the loanword (from the English language) ``beaten'' [(to) make beat music] is acceptable according to \cite{Duden1}, and it is homophonic to ``bieten'' [(to) offer].
\item[(y)] The word ``toi'' from the saying ``toi, toi, toi'' [break a leg] sounds like ``Toy'' [sex toy]. Alternatively, a ``Bayer'' [Bavarian] can be spelled ``Baier''.
(We could also invoke the ambiguous spelling of `Yoghurt' and `Joghurt'. For yet another possibility, the Hungarian word ``Gulyas'' [goulash] has been assimilated also as ``Gulasch''.)
\item[(j)] As to `j', we use the word ``Yak'' [yak]  (from the Tibetan `gyag') and its similarity to ``Jacke'' [jacket] which are not homophonic as such, but their respective diminutives ``Y\"akchen'' and ``J\"ackchen" (note the ensuing umlaut for either case) are.
\item[(q)] Finally, for the quite rare letter `q', we can use the French word ``clique'' (which has been adapted into German with a short `i'), whose pronunciation agrees with that of ``klicke'' [(I) click]. Another possibility is to note that the letter ``Q'' itself can be used as a word (say, as the Q in a game of Scrabble) and is homophonic to ``Kuh'' [cow]).
\end{enumerate}

In the table below we successively eliminate the letters on the left using the homophonic ambiguity displayed on their right, completing the proof of the theorem.
\smallskip\noindent{\small
\halign{\hskip3truecm#&$#|$&\quad#\hfill\cr
a&&Waage --- wage\cr
h&&Wahl --- Wal\cr
e&&Meer --- mehr\cr
o&&Boot --- bot\cr
g&&Bug --- buk\cr
b&&Alb --- Alp\cr
n&&Mann --- man\cr
v&&viel --- fiel\cr
w&&wage --- vage\cr
l&&gewallt --- Gewalt\cr
r&&starrt --- Start\cr
f&&schafft --- Schaft\cr
p&&klappst --- klapst\cr
s&&fasst --- fast\cr     
t&&Zittern --- Zithern\cr
d&&Stadt --- statt\cr
m&&hemmt --- Hemd\cr
c&&packt --- Pakt\cr
z&&Kitz --- Kitts\cr
x&&lax --- Lachs\cr
k&&klipp --- Clip\cr
u&&tschau --- ciao\cr
i&&roien --- reuen\cr   
y&&Toy --- toi\cr   
j&&J\"ackchen --- Y\"akchen \cr
q&&Clique --- klicke. \qquad \qquad \hfil $\Box$\cr}}

\bigskip
{\bf Generalizations.} One can also try to include the umlaute \"a, \"o, \"u, and the ``sharp s'' \ss\ into these investigations. The result remains the same. 
Our suggestion for the corresponding trivializations are the following: for `\"a' we invoke that  in combination with `u' the diphthongs `\"au' and `eu' sound alike, for instance in the words ``h\"autig'' [of a skinny texture] and ``heutig'' [contemporarily]; alternatively, we can use that a long `\"a' can sound like  the  `ai' for certain loanwords from the English language, for example in ``F\"ahre'' [ferry] and ``faire'' [fair]; for `\"o' we use that certain words are spelled with both the original French `eu' and the assimilated German `\"o', like ``Fris\"or'' and ``Friseur''; furthermore, the pronunciation of `\"u' is often the same as that of `y', like in the Greek letter ``My'' [mu] and ``m\"uh'' [(I) labour], or, a far better one due to Martin Brandenburg, ``Mythen'' [myths] and ``m\"uhten'' [(they) laboured]. Finally, a `sharp s' at the end of a word is typically preceded by a long vowel, and hence it is not difficult to construct word pairs like ``a\ss'' [(I) ate] and ``Aas'' [(rotten) carcass].

\medskip
\halign{\hskip3truecm#&&\quad#\hfill\cr
\"a&&h\"autig --- heutig\cr                    
\"o&&Fris\"or --- Friseur\cr
\"u&& m\"uh -- My\cr                
\ss&&a\ss\  --- Aas.\cr}

\section{Korean}
\label{S:Korean}

What differentiates Korean from the languages discussed in \cite{QHGL} is the number of alphabets and some fundamental grammar structures. Nevertheless, there exist many rules regarding homophones, so the first natural assumption would be that the resulting quotient group shouldn't have too many elements. It turns out that this is indeed the case. 

Here we note that this mathematical analysis of Korean does not describe the entire structure of the Korean language. It takes the phonetic aspect of the language and restructures the alphabets into a free group with a very specific and somewhat restrictive equivalence relation. Using such a structure, we inevitably lose a lot of information about Korean language, but are, however, rewarded with a unique finite group that characterizes it.

Now, let us begin with describing some necessary concepts about the Korean language.

\subsection{Some Basics of Korean}

Korean characters, like English, consist of {\it vowels} and {\it consonants}. The alphabet contains 19 consonants and 21 vowels. The exact list is shown below in Table \ref{T:KoreanCharacters}.

\begin{table}[h!]
\centering
\begin{tabular}{|c|c|}\hline Consonants & ㄱㄲㄴㄷㄸㄹㅁㅂㅃㅅㅆㅇㅈㅉㅊㅋㅌㅍㅎ \\ \hline Vowels & ㅏㅐㅑㅐㅓㅔㅕㅖㅗㅘㅙㅚㅛㅜㅝㅞㅟㅠㅡㅢㅣ \\ \hline \end{tabular}
\caption{Korean characters \cite{KLI}.}
\label{T:KoreanCharacters}
\end{table}

Because of the complications arising from the unique structure of Korean, from here on, each of the above symbols in the table will be called {\it characters}. To show why such clarification is crucial, let us take a look at a Korean word that stands for ``number''. It is written as 수. This word is comprised of a single letter, and that letter is comprised of a consonant and a vowel, which are, in this case, ㅅ and ㅜ. These letters form the bases of Korean words, as no single consonant or vowel is ever used alone without the other. However, this is not the end.

To add to the already complex structure, a single letter can be made up of multiple consonants and a vowel, up to 3 consonants and one vowel. Denoting vowels and consonants as $v$ and $c$ in respective order, the possible combinations are, $\{c+v, c+v+c, c+v+c+c\}$. From here on forth, expressions of the form $c+v+\cdots$ will be called the ordered decomposition. The fact that these are the only combinations, however, effectively erases the need to distinguish between letters and combinations of characters. We present the needed argument below. 

\begin{theorem}
Ordered decomposition uniquely encodes any formal composition of Korean letters or words. Or equivalently, the formal expression of a Korean word is uniquely encoded in the ordered decomposition. 
\end{theorem}

\begin{proof}

\normalfont For the base case, a word comprised of one letter is always trivially identifiable in Korean because it will be one of the form $c+v$, $c+v+c$, or $c+v+c+c$. It is clear that none of the three forms can be divided up to form a word of two letters.

Now, suppose that a word comprised of n letters is uniquely identifiable. Now the final word in that letter can either end with a consonant or a vowel. If the word of n letters ended on a vowel, the possible $n+1^{th}$ letter after that vowel can be of the form $c+v$, $c+v+c$, or $c+v+c+c$. The three possibilities of this word with the $n+1^{th}$ letter is $c+\cdots +v+c+v$, $c+\cdots +v+c+v+c$, or $c+\cdots +v+c+v+c+c$. In any case, it is impossible to divide up $v+c+v$, $v+c+v+c$, or $v+c+v+c+c$ to end with a letter other than to divide between the first two $v$ and $c$. The same argument holds for the case when the $n^{th}$ letter ended on a consonant.
\end{proof}
 
Now that we've established some basics we will examine the homophonic structure of the quotient group $G$ of the free group on 40 Korean characters, given by the equivalence relation A=B whenever A and B have the same pronunciation in Korean. We will use a standard pronunciation guide such as \cite{KLI} for reference.

\subsection{Triviality of Consonants}

We first show that all consonants are trivial. We do this in three steps. 

\subsubsection{ ㅇ is trivial}

To show this, let us take a look at the word 안일하다 (to be idle). 안일하다 has exactly the same pronunciation as 아닐하다 \cite{KLI}. Just by looking at the two words, 하다 is present on both sides, so it can be canceled out. Now the equivalence relation is between ㅇㅏㄴㅇㅣㄹ and ㅇㅏㄴㅣㄹ. Clearly after canceling out, ㅇ=1, and hence ㅇ is trivial.

\subsubsection{ ㄱ=ㄲ=ㅋ, ㄷ=ㅅ=ㅆ=ㅈ=ㅊ=ㅌ, ㅂ=ㅍ}

To show the above equivalence relations, let us examine the words containing letters of the form $c+v+c$. First, we will examine the words 부엌 (kitchen), 밖 (outside). By the equivalence relation defined above, 부엌=부억 and 밖=박. By rewriting these relations in ordered decomposition, ㅂ+ㅜ+ㅇ+ㅓ+ㅋ=ㅂ+ㅜ+ㅇ+ㅓ+ㄱ and ㅂ+ㅏ+ㄲ=ㅂ+ㅏ+ㄱ. Now it is clear that ㄱ=ㅋ and ㄱ=ㄲ. By the transitive property ㅋ=ㄲ=ㄱ.

For the second part we can examine the equivalence relations, 낫 (scythe) =낟, 낮 (day) =낟, 낯 (face) =낟, 밭 (field) =받. Clearly ㄷ=ㅅ=ㅈ=ㅊ=ㅌ.  Proving ㄷ=ㅆ is a bit more difficult as there are no single lettered words in Korean ending in ㅆ. To prove this we need to look at a two lettered word 불소 (fluorine). By the equivalence relation 불소=불쏘, and clearly ㅅ=ㅆ. Since we already know ㅅ=ㄷ, by the transitive property, ㄷ= ㅆ, thus concluding our proof of the second equivalence relation. 

For the last equivalence relation, we can look at 짚 (hay) = 집, and can conclude that ㅂ = ㅍ.

By proving these relations, we've reduced the set of consonants into $\{$ㄱ,ㄴ,ㄷ,ㄸ,ㄹ,ㅁ,ㅂ,ㅃ,ㅉ,ㅎ$\}$.

\subsubsection{Consonants are trivial}

To further reduce the set of consonants let us look at the equivalence relation 앞마당 (lawn) = 암마당. This shows that ㅁ = ㅍ, and ㅍ = ㅂ so, ㅁ = ㅂ. Also 있는 (existing)  = 인는, and so ㄴ = ㅆ = ㄷ. Also, 국물 (soup)  = 궁물, and 놓는 (lay down)  = 논는, so ㄱ is trivial and ㄷ = ㄴ = ㅎ. Observe that 숱하다 (to be in abundence)  = 수타다, which shows that ㅎ is also trivial. Since ㄷ = ㅎ and ㅎ is trivial, ㄷ is also trivial. Now, there only remain five nontrivial consonants, $\{$ㄸ,ㄹ,ㅂ,ㅃ,ㅉ$\}$.

Let's look at the equivalence relation 웃다 (smile)  = 욷따, which in ordered decomposition is, ㅇ+ㅜ+ㅅ+ㄷ+ㅏ = ㅇ+ㅜ+ㄷ+ㄸ+ㅏ. We know ㅇ,ㅅ = ㄷ are trivial, so ㅜ+ㅏ = ㅜ+ㄸ+ㅏ. Hence ㄸ = 1 and so ㄸ is also trivial. 약지 (ring finger)  = 약찌, hence ㅉ = ㅈ = ㄷ = 1 and ㅉ is trivial. 막론 (whether)  = 망논, which can be rewritten as ㅁ+ㅏ+ㄱ+ㄹ+ㅗ+ㄴ = ㅁ+ㅏ+ㅇ+ㄴ+ㅗ+ㄴ, and since ㄱ,ㄴ = ㄷ are both trivial, ㅁ+ㅏ+ㄹ+ㅗ = ㅁ+ㅏ+ㅗ, and so ㄹ is trivial. 국밥 (soup and rice)  = 국빱 implies that ㅂ = ㅃ. There remains only one non-trivial consonant, ㅂ.

Lastly we need to examine a word with a letter of the form $c+v+c+c$. 넓다 (wide)  = 널따, and we know that ㄴ,ㄹ,ㄷ,ㄸ are all trivial. So, after canceling both sides, we have ㅓ+ㅂ+ㅏ = ㅓ+ㅏ, and so ㅂ is trivial. Hence we've proved the triviality of all Korean consonants. 

\subsection{Vowels have 2 Non-trivial Elements: Vowels =  $\{$ㅏ,ㅗ$\}$}

While the in examining consonants we only needed to look at a single equivalence relation, vowels are not so easy. There are multiple equivalence relations between 3 or more vowels, so we need to sort through these relations to see how they can be reduced. 

Let us begin examining the vowels by the equality 가지어 (have)  = 가져 = 가저. All consonants are trivial, so we can simply look at ㅏ+ㅣ+ㅓ = ㅏ+ㅕ = ㅏ+ㅓ. From this we can conclude that ㅣ is trivial and ㅓ = ㅕ. Also, 통계 (statistics)  = 통게 implies that ㅖ = ㅔ. 희미 (faint)  = 히미 implies ㅢ = ㅣ. Since ㅣ is trivial ㅢ is trivial. 금괴 (gold bar)  = 금궤 (metal box) implies ㅚ = ㅞ. Also, it is allowed that 위 is pronounced as 우+이, and so ㅟ = ㅜ.(누이다 = 뉘다) 되어 = 돼 implies ㅚ+ㅓ = ㅗ+ㅓ = ㅙ = ㅗ+ㅐ. 싸이다 = 쌔다 implies ㅏ+ㅣ = ㅏ = ㅐ. 트이다 = 틔다 = 티다 implies that ㅡ is also trivial. 미아 = 먀 implies ㅏ = ㅑ. 쏘이다 = 쐬다 implies ㅗ+ㅣ = ㅗ = ㅚ. Also, ㅚ = ㅞ, so ㅗ = ㅚ = ㅞ.

Listing these rules into an easily decipherable form we get:

%
%
%
%
%
%
%
%
%
%
%
%
%
%
%
%
%
%

\begin{multicols}{2}{
\begin{enumerate}
\item ㅣis trivial

\item ㅓ=ㅕ

\item ㅖ=ㅔ

\item ㅢ=ㅣ

\item ㅚ=ㅞ

\item ㅟ=ㅜ+ㅣ $\Leftrightarrow$ ㅟ=ㅜ

\item ㅚ+ㅓ=ㅙ

\item ㅏ+ㅣ=ㅐ $\Leftrightarrow$ ㅏ=ㅐ

\item ㅡ+ㅣ=ㅣ $\Leftrightarrow$ ㅡ=1

\item ㅣ+ㅏ=ㅑ $\Leftrightarrow$ ㅏ=ㅑ

\item ㅗ+ㅣ=ㅚ $\Leftrightarrow$ ㅗ=ㅚ

\item ㅢ=ㅔ $\Leftrightarrow$ ㅔ=1

\item ㅗ+ㅐ=ㅙ

\item ㅝ=ㅜ+ㅓ

\item ㅛ=ㅣ+ㅗ 

\item ㅠ=ㅣ+ㅜ 
\end{enumerate}}
\end{multicols}

\vspace{1pc}

Just by looking at these rules, we can reduce the the set $\{$ㅏㅐㅑㅐㅓㅔㅕㅖㅗㅘㅙㅚㅛㅜㅝㅞㅟㅠㅡㅢㅣ$\}$ of vowels into $\{$ㅏㅓㅗㅘㅙㅛㅜㅝㅞㅠ$\}$. 

Now we outline the rest of the process:

\begin{itemize}
\item
13) and 8) combine to show that ㅗ+ㅐ=ㅗ+ㅏ=ㅘ=ㅙ, implying ㅘ=ㅙ.
\item
5) and 11) combine to show that ㅗ=ㅞ=ㅜ+ㅔ, and since 12) stated that ㅔ is trivial, ㅗ=ㅜ.
\item
As a direct result of ㅗ=ㅜ, 14), 13), 11), 7) and 1), ㅝ=ㅜ+ㅓ=ㅗ+ㅓ=ㅗ+ㅣ+ㅓ=ㅚ+ㅓ=ㅙ=ㅗ+ㅐ. So ㅝ=ㅙ=ㅘ.
\item 
1) and 15) together shows that ㅗ=ㅛ.
\item
Since ㅠ=ㅣ+ㅜ and we've concluded that ㅗ=ㅜ, ㅠ=ㅣ+ㅜ=ㅣ+ㅗ=ㅗ.
\item
Recall that from 7) and 13), we have ㅚ+ㅓ=ㅙ=ㅗ+ㅐ. However, ㅚ=ㅗ+ㅣ=ㅗ, so ㅗ+ㅓ=ㅗ+ㅐ, and by cancelling ㅗ, we have that ㅓ=ㅐ.
\item
8) states that ㅏ=ㅐ, so with the above result, ㅏ=ㅓ. 
\item
Since ㅘ=ㅗ+ㅏ, ㅘ is generated by ㅗ and ㅏ.
\end{itemize}

\noindent Ultimately, after all iterations of the 16 rules, we are left with a free group that is generated by the free generating set $\{$ㅏ,ㅗ$\}$. 

%
%
%

\section{Turkish}
\label{S:Turkish}

In this section we determine the homophonic quotient group for Turkish. There are several Turkic languages and alphabets encoding them have many commonalities. We will exclusively focus on the Modern Turkish alphabet used in Turkey to encode Turkish. 

\subsection{The Sounds of Turkish}

The Modern Turkish alphabet was introduced in 1928 along with a wide-reaching literacy campaign. The Latin-based script was developed to replace the use of the Arabic script, and contains a total of 29 letters (8 vowels and 21 consonants) as seen in Table \ref{T:TurkishLetters}:

\begin{table}[h!]
\centering
\begin{tabular}{|c|c|}
\hline
Consonants & b \; c \; \c{c} \; d \; f \; g \; \u{g} \; h \; j \; k \; l \; m \; n \; p \; r \; s \; \c{s} \; t \; v \; y \;z\\\hline
Vowels& a \; e \; {\i} \; i \; o \; \"{o} \; u \; \"{u} \\
\hline
\end{tabular}
\caption{Letters of the Modern Turkish script (only lowercase letters are given).}
\label{T:TurkishLetters}
\end{table}

This set of letters was specifically selected to represent the sounds present in the spoken language of the time, taking the Istanbul dialect as the standard. Each letter is supposed to represent a unique sound of the spoken language (except the so-called {\it soft g}, {\it \u{g}}, which tends to extend the vowel before it and blends it to a following vowel if there is one, but is otherwise completely silent; see \cite{Fuchs} for more on the soft g). For more on the sound system of Modern Turkish, see \cite{Yavuz}. 

To this day the Modern Turkish script retains most of its phonetical representativeness \cite{KY}. Indeed many hold that there are no homophones in Turkish; see for instance \cite{RW} where Turkish is described as a ``completely transparent writing system" with ``invariant and context-independent one-to-one mappings between orthography and phonology". 

This suggests that the free group generated by the 29 sound representatives will not shrink  much if at all when we try introducing homophonic equivalences. Nonetheless there are indeed some relations we might use if we consider ``how words are actually pronounced by real live people".\footnote{In his MathSciNet review (MR1273406 (95e:00027)), James Wiegold notes that the authors of \cite{QHGL} ``have [perhaps deliberately?] neglected all considerations of how words are actually pronounced by real live people.'' Clearly if we were to take into consideration each native speaker's distinct pronunciation patterns, the homophonic quotients problem would become quite intractable. However we will indeed introduce some of this complication into our analysis of Turkish. This may be justified by the fact that there is deemed to be a standard spoken Turkish, and it is indeed distinct from most formal descriptions of the orthography / phonology correspondence for the language.}


\subsection{The soft g disappears}

As noted above the soft g is often not a distinctly pronounced consonant but instead helps to accentuate or blend the surrounding vowels. Most native speakers would agree that we can identify the following encodings of the male name meaning ``Khan":
\begin{center}
 Kaan = Ka\u{g}an.
\end{center}
Thus in the quotient group we would identify the soft g with identity. 

\subsection{Other disappearance acts: h and t}
The standard pronunciation of the word ``dershane" (classroom) overlaps with the pronunciation of ``dersane", thus allowing us to conclude that $h$ too is trivial in the quotient. Similarly the double t's in the words ``Hacettepe" and ``An{\i}ttepe" (two location names in Ankara) are most commonly pronounced as if they were written as ``Hacetepe" and ``An{\i}tepe" respectively. Thus we can identify {\it tt} with {\it t}, trivializing {\it t}.

\subsection{Vowel Confusion: the transformations of {\it a} and {\it e} into {\it \i} and {\it i} and two final disappearance acts}

The Turkish language captures the phrase ``let me look" in the single word ``bakay{\i}m". The native speaker pronounces the latter in the same way that she would read the letter collection ``bak{\i}y{\i}m". This allows us to identify $a = \i$. 
Similarly the phrase ``i\c{c}ecek" (drink) is pronounced the same way that one would read ``i\c{c}icek" and so we identify $e = i$. 

Finally the word ``a\u{g}abey" for older brother has an almost universally accepted informal spelling, ``abi", representing the way people actually pronounce the word. This gives us two  additional trivializations, of $a$ and $y$. 

Putting the above reductions together we conclude that the homophonic quotient group for Turkish is a free group on 22 generators:
\begin{center}
b \; c \; \c{c} \; d \; e (= i) \; f \; g \; j \; k \; l \; m \; n \; o \; \"{o} \; p \; r \; s \; \c{s} \; u \; \"{u} \; v \;z
\end{center}

\section{Final Words: Bringing the three threads together}

\parskip=3pt

In this paper we investigated three different languages and their writing systems. We believe that our results offer a neat example of applied algebra. Algebraic structures have been identified in various symmetrical constructions of nature such as crystals as well as in a range of sociological and anthropological contexts such as the kinship structure of the Warlpiri of Australia.\footnote{As ethnomathematician Marcia Ascher describes in detail in her book, the kinship structure of the Warlpiri, an indigeneous people in Australia, can be accurately and succinctly represented by the dihedral group of order $8$. See Chapter 3 of \cite{Ascher} for details.} In this note, we explored how the writing system of a modern language and its correspondence with the sounds of that language can be encoded in group theory. 

It is important to note that our methods do not address the full phonetic structure of any single language. Our work only pertains to the relationship between orthography and phonology of a language, that is, the extent to which a single symbol may represent a multiplicity of sounds of a given language. A simplistic interpretation of our method would suggest that if the generating set for the resulting quotient group is small, there are, on average, more sounds represented by a single symbol. 

We should also note that the complexity of the resulting group may be correlated not directly with the complexity of the sound system of a given language but perhaps more with the maturity of the particular writing system associated to it. Languages evolve, and oral traditions evolve much faster than written ones. Thus a young script like Modern Turkish might be naturally more representative of the phonetical structure of the language and equivalently offer fewer homophones than a script which is more mature, such as the Korean one, which in turn may offer fewer homophones than an even older script such as the German one.

\smallskip

\noindent {\bf Acknowledgment.}
The authors thank the reviewer for helpful suggestions.

\end{document}